\documentclass[final,leqno,onefignum,onetabnum]{siamltex}
\usepackage{psfrag}
\usepackage{epstopdf}
\usepackage{color}
\usepackage{amscd}
\usepackage[official]{eurosym}
\usepackage{multirow}
\usepackage{epsfig}
\usepackage{color, float}
\usepackage{graphics}
\usepackage{graphicx}

\usepackage{amsmath}
\usepackage{amssymb}
\usepackage{mathrsfs}
\usepackage{amsfonts}
\usepackage{cite}

\newtheorem{thm}{Theorem}[section]
\newtheorem{prop}[thm]{Proposition}

\newtheorem{lem}[thm]{Lemma}

\newtheorem{defn}[thm]{Definition}

\newtheorem{rmk}[thm]{Remark}

\numberwithin{equation}{section}

\makeatletter % `@' now normal "letter"
\@addtoreset{equation}{section}
\makeatother  % `@' is restored as "non-letter"

\newcommand\norm[1]{\lVert#1\rVert}
\newcommand\abs[1]{\lvert#1\rvert}

\newcommand\RR{\ensuremath{\mathbb{R}}}

\def\tint{\text{int}}
\newcommand{\ucite}[1]{\cite{#1}}

\def \a{\alpha}  \def \g{\gamma} \def \d{\delta}
\def \t{\theta}   \def \e{\epsilon}
\def \s{\sigma}   \def \o{\omega}

\def\R{\mathbb{R}}
\allowdisplaybreaks

\title{Monotone semiflows with respect to high-rank cones on a Banach space}

\author{Lirui Feng\thanks{School of Mathematical Science, University of Science and Technology of China, Hefei, Anhui, 230026, People¡¯s Republic of China (ruilif@mail.ustc.edu.cn). Supported by China Scholarship Council.}
\and Yi Wang\thanks{Wu Wen-Tsun Key Laboratory, School of Mathematical Science, University of Science and Technology of China, Hefei, Anhui, 230026, People¡¯s Republic of China (wangyi@ustc.edu.cn). Partially supported by NSF of China No.11371338, 11471305, and the Fundamental Research Funds for the Central Universities.}
\and Jianhong Wu\thanks{Department of Mathematics and Statistics, York University, Toronto, ON, M3J 1P3, Canada (wujh@mathstat.yorku.ca). Supported by the Natural Science and Engineering Research Council of Canada and the Canada Research Chairs program.}}

\begin{document}

\maketitle

\begin{abstract}
We consider semiflows in general Banach spaces motivated by monotone cyclic feedback systems or differential equations with integer-valued Lyapunov functionals.  These semiflows enjoy strong monotonicity properties with respect to cones of high ranks, which imply order-related structures on the $\omega$-limit sets of precompact semi-orbits. We show that for a pseudo-ordered precompact semi-orbit the $\omega$-limit set $\Omega$  is either ordered, or is contained in the set of equilibria, or possesses a certain ordered homoclinic property. In particular, we show that if $\Omega$ contains no equilibrium, then $\Omega$ itself is ordered and hence the dynamics of the semiflow on $\Omega$ is topologically conjugate to a compact flow on $\mathbb{R}^k$ with $k$ being the rank. We also establish a Poincar\'{e}-Bendixson type Theorem in the case where $k=2$. All our results are established without the smoothness condition on the semiflow, allowing applications to such cellular or physiological feedback systems with piecewise linear vector fields and to such infinite dimensional systems where the $C^1$-Closing Lemma or smooth manifold theory has not been developed.
\end{abstract}


\pagestyle{myheadings}
\thispagestyle{plain}
\markboth{LIRUI FENG, YI WANG, AND JIANHONG WU}{MONOTONE SEMIFLOWS}

\section{Introduction}

We consider the global dynamics of  a {\it continuous semiflow $\Phi_t$ monotone with respect to a cone $C$ of rank-$k$ on a Banach space $X$}, here roughly speaking, a cone $C$ of rank-$k$ is a closed subset of $X$ that contains a linear subspace of dimension $k$ and no linear subspaces of higher dimension.  Krasnoselskii et al. \cite{KLS} introduced the {\it cones of rank-$k$} and obtained a generalized Krein-Rutman theory associated with these high-rank cones, see also Fusco and Oliva \cite{FO1} for important discussions in the finite dimensional case.

For a given convex cone $K$, $K\cup (-K)$ defines a cone of rank-$1$. Therefore, our considered a class of semiflows includes the order-preserving (monotone) semiflows well studied since the pioneering work of Hirsch \ucite{Hirsch1,foo2,foo3,foo4,foo5,foo6}.  An essential difference between a convex cone $K$ and a high-rank cone is the lack of convexity in the high-rank cone. This difference makes the study of dynamics of semiflows with respect to a high-rank cone a challenging task.

An important example of a continuous semiflow monotone with respect to a cone of rank-$2$ is the following monotone cyclic feedback system
$$
\dot x^i=f^i(x^i, x^{i-1}),\quad i (\textnormal{mod } n),
$$
where each $f^i$ is $C^1$-smooth, and the variable $x^{i-1}$ forces $\dot x^i$ monotonically, namely
$\delta ^i\frac{\partial f^i(x^i, x^{i-1})}{\partial x^{i-1}}>0$ with some $\delta ^i\in \{1, -1\}$. We have either inhibition if $\delta ^i<0$ or excitation if $\delta ^i>0$, and the entire system is said to have a negative  (if $\Pi _{i=1}^n\delta ^i<0$) or positive (if $\Pi _{i=1}^n\delta ^i>0$) feedback. It has been shown that the Poincar\'{e}-Bendixson theorem holds for such a monotone cyclic system (see \cite{M-PSmi90} and references therein), and this remains true even for its extensions involving delays in the feedback \cite{M-PSe96,M-PSe96-2,M-PN}. These monotone cyclic systems arise very naturally from cellular, neural and physiological control systems. However, it should be remarked that in these systems, specially those arising from additive neural network models where the feedback function is piecewise constant (binary on-or-off) or piecewise linear \cite{KaplanGlass, Milton,Wu}, the monotonicity condition  $\delta ^i\frac{\partial f^i(x^i, x^{i-1})}{\partial x^{i-1}}>0$ is normally replaced by strict monotonicity without involving the derivative, the semiflow is still monotone but is no longer $C^1$-smooth. Thus  arguments based on the $C^1$-Closing Lemma and/or invariant manifold theories cannot be applied to such semiflows even in finite dimensional spaces. One of the contributions of our work is to establish the Poincar\'{e}-Bendixson theorem for general semiflows monotone with respect to high-rank cones defined on a general Banach space without the $C^1$- smoothness condition on the semiflows. This contribution is also significant for other classes of dynamical systems with integer-valued Lyapunov functionals including scalar parabolic equations on an interval or a circle \cite{Fied89,JR,SY,SWZ,Te}, and competitive-cooperative tridiagonal systems \cite{FGW13,Fusco87,LW,Smillie,Smith91}. Other models arising from important applications, to which our results can be applied, include $n$-dimensional competitive dynamical systems (see \cite{foo3,HirSm,Smi95,WJ} and references therein) which can be viewed as strongly monotone systems with respect to the cone $C$ of rank-$(n-1)$ whose complemented cone is of rank-$1$.

We will call a semiflow, that is monotone with respect to a particular order defined by a convex cone (recall that this generated a semiflow with respect to a cone of rank-$1$), a {\it classical} monotone semiflow. For the classical monotone semiflow, the pioneering work by Hirsch \ucite{Hirsch1,foo2,foo3,foo4,foo5,foo6} showed that a precompact semi-orbit generically approaches the set of equilibria (referred as generic quasi-convergence). Subsequent studies conclude that for a classical smooth strongly monotone system, precompact semi-orbits are generically convergent to equilibria in the continuous-time case \cite{HirSm,Smi95} or to cycles in the discrete-time case  \cite{Po1,Po2}. There are exactly two different kinds of nontrivial (not an equilibrium) semi-orbits for the classical monotone systems: {\it pseudo-ordered semi-orbits} and {\it unordered semi-orbits}. A nontrivial semi-orbit $O^+(x):=\{\Phi_t(x):t\ge 0\}$ is called {\it pseudo-ordered} if $O^+(x)$ possesses one pair of distinct ordered-points $\Phi_{t}(x),\Phi_{s}(x)$ (i.e.,  $\Phi_{t}(x)-\Phi_{s}(x)\in K\cup (-K)$); while $O^+(x)$ called {\it unordered} if any pair of distinct points in $O^+(x)$ is unordered (see Definition \ref{D:two-type-orbit} with $C=K\cup (-K)$). In the classical monotone semiflows, every pseudo-ordered precompact orbit converges to equilibrium due to the Monotone Convergence Criterion for $\Phi_t$ (see, e.g. \cite[Theorem 1.2.1]{Smi95}), while any unordered orbit can be projected over a certain $1$-codimensional hyperplane outside $K\cup (-K)$. A geometrical insight of the structure of $K\cup (-K)$ yields that it consists of $1$-dimensional linear subspaces and contains no higher dimensional subspaces. As a consequence, as long as $\Phi_t(x)$ is pseudo-ordered, one may project $\Phi_t(x)$ onto a straight line so that the corresponding dynamics is essentially the same as that in an $1$-dimensional system, which makes the Monotone Convergence Criterion quite natural.

Since the Monotone Convergence Criterion plays a key role in the classical monotone systems with respect to $K$, one naturally wonders if it still holds for monotone systems with respect to a high-rank cone. Recently, Sanchez \cite{San09,San10} tackled this problem for  monotone flows $\Phi_t$ on $\mathbb{R}^n$ with respect to a high-rank cone $C$. By using  the $C^1$-Closing Lemma, he proved that if $\Phi_t$ is smooth and $C$-cooperative (see \cite[Definition 6]{San09}), which is stronger than strong monotonicity of the flow, then the closure $\Omega_1$ of any orbit in the omega-limit set $\Omega$ of a pseudo-ordered orbit is ordered with respect to $C$, and hence, the corresponding flow on $\Omega_1$ is essentially $k$-dimensional. The ordering property of the omega-limit set $\Omega$ itself, however, remains unknown (see \cite[p.1984]{San09}).

Our focus here, for a {\it continuous monotone semiflow $\Phi_t$} with respect to a cone $C$ of rank-$k$ on a Banach space $X$, is the ordering property of the omega-limit set $\Omega$ of a pseudo-ordered orbit. We will first show that if $\Phi_t$ is strongly monotone with respect to $C$, then the closure of any orbit in $\Omega$ is ordered (see Theorem A). As we mentioned above, for the finite dimensional case $X=\mathbb{R}^n$, Sanchez \cite{San09} has obtained this result for $C^1$-smooth flows by using the Closing Lemma. However, to the best of our knowledge, the Closing Lemma is not established in general infinite-dimensional spaces. Moreover, even in the case of $X=\mathbb{R}^n$, the results of Sanchez based on the use of Closing Lemma cannot be applied to our setting where the smoothness is not imposed for the semiflow. As a consequence, our results are novel even for the finite-dimensional case, and this gives an affirmative answer to the question posed in \cite[Remark 3]{San09}.

We further examine the ordering property of the omega-limit set $\Omega$ itself. We obtain the following {\it trichotomy} result (see Theorem B):
  \vskip 1mm

(i)   Either $\Omega$ is ordered, i.e., $p-q\in C$ for any $p,q\in \Omega$;

(ii)  Or $\Omega\subset E$ is unordered, where $E$ is the set of all the equilibria of $\Phi_t$;

(iii)  Or $\Omega$ possesses an ordered homoclinic property, that is, there is an ordered and invariant subset $\tilde{B}\subsetneq\Omega$ such that, for any $p\in \Omega\setminus \tilde{B}$,  it holds that $\omega(p)\cup\alpha(p)\subset \tilde{B}$ and $\alpha(p)\subset E$.

\vskip 1mm

\noindent An immediate consequence of the trichotomy is that $\Omega$ itself is ordered if $\Omega$ contains no equilibrium. This partially solves the problem posed in \cite[Line 15-16, p.1984]{San09} even for the infinite dimensional case. As a consequence, the dynamics on $\Omega$ is topologically conjugate to a compact flow on $\mathbb{R}^k$ when $\Omega$ contains no equilibrium.

When $k=2$ and $X=\mathbb{R}^n$, Sanchez \cite{San09} further concluded that if $\Omega$ contains no equilibrium, then $\Omega$ itself is an ordered closed orbit, a conclusion that we will refer as to the Poincar\'{e}-Bendixson Theorem for $\Phi_t$. The crucial tools in his proof are the generalized Perron-Frobenius Theorem (see \cite{FO1}) and theory of invariant manifolds in $\mathbb{R}^n$, which again strongly depend on the $C^1$-smoothness assumption on $\Phi_t$. In Section 5, we remove this smoothness assumption (see Theorem C), based on the  approach motivated by \cite{Smi95} with the use of the chain-recurrent property of $\Omega$.

Finally, we also would like to mention the work by Smith \cite{SmithR-2,SmithR-3} that established a Poincar\'{e}-Bendixson theorem for systems of ordinary differential equations possessing a certain quadratic Lyapunov function. We will show that our Theorem C is indeed a generalization of the Poincar\'{e}-Bendixson theorem of Smith in \cite{SmithR-2,SmithR-3}. Under the smoothness assumption, Sanchez \cite{San09} has successfully established the connections between Smith's results and the strongly monotone systems with respect to some cones of rank-$2$. Our work also shows that such a connection remains valid even for a locally Lipschitz continuous vector field which is not necessarily defined in the whole phase space.

The paper is organized as follows.  In Section 2 we will introduce some notations and definitions, and summarize some established facts for cones of rank-$k$ and the strongly monotone continuous semiflows with respect to such high-rank cones. We will present the main results in this section but defer the detailed proofs to Sections 3-5. In Section 3, by using the non-wondering property, we prove that the closure of any orbit in the omega-limit set $\Omega$ of a pseudo-ordered orbit is ordered (see Theorem A). Based on this, in Section 4, we focus on the ordering property of $\Omega$ itself and prove the corresponding trichotomy of $\Omega$ (see Theorem B). In particular,  when $\Omega$ does not contain any equilibrium, we show that the dynamics on $\Omega$ is topologically conjugate to a compact flow on $\mathbb{R}^k$. For $k=2$, we prove, in Section 5, the Poincar\'{e}-Bendixson Theorem (Theorem C). Finally, in Section 6, we discuss the relationship between strongly monotone continuous systems and other well-known systems mentioned above, as well as a generalization of the extended Poincar\'{e}-Bendixson Theorem of R. A. Smith.

\section{Notations and main results}

We start with some notations and a few definitions. Let $(X,\norm{\cdot})$ be a Banach space. We first define a cone of rank-$k$.

\begin{defn}\label{D:k-cone}
 A closed set $C\subset X$ is called a cone of rank-$k$ ({\it abbr. $k$-cone}) if the following
 are satisfied:

 i) For any $v\in C$ and $l\in \RR,$ $lv\in C$;

 ii)  $\max\{\dim W:C\supset W \text{ linear subspace}\}=k.$
\end{defn}

\noindent Roughly speaking, a $k$-cone $C\subset X$ contains a linear subspace of dimension $k$ and no linear subspaces of higher dimension. To the best of our knowledge, $k$-cone was independently introduced  by Krasnoselskii et al. \cite{KLS} and Fusco \& Oliva \cite{FO1}. In particular, given any traditional convex cone $K\subset X$ (see, e.g. \cite{Smi95}), it is clear that $C=K\cup (-K)$ defines a $1$-cone. Other concrete examples of $k$-cones can be found in \cite{KLS,LW,San09,FO1,M-PSe96,M-PSe96-2,M-PN, SmithR-1,SmithR-2,SmithR-3,Te}. The essential difference between $k$-cones and the $1$-cone $K\cup (-K)$ is the lack of convexity.

A $k$-cone is $C\subset X$ is said to be {\it solid} if the interior $\tint C\ne \emptyset$; and $C$ is called {\it $k$-solid} if there is a $k$-dimensional linear subspace $W$ such that $W\setminus \{0\}\subset \tint C$. Given a $k$-cone $C\subset X$, we call $C$ is {\it complemented} if there exists a $k$-codimensional space $H^{c}\subset X$ such that $H^{c}\cap C=\{0\}$.

For two points $x,y\in X$, we say that {\it $x$ and $y$ are ordered, denoted by $x\thicksim y$}, if $x-y\in C$. Otherwise, $x,y$ are called to be {\it unordered}, which we denote by $x\rightharpoondown y$. The pair $x,y\in X$ are said to be {\it strongly ordered}, denoted by $x\thickapprox y$, if $x-y\in \tint C$.
For sets $A,B$ we write $A\thicksim B$ if $x-y\in C$ for any $x\in A$ and $y\in B$.

A subset $W\subset X$ is {\it called ordered} if $x\thicksim y$ for any $x,y\in W$. $W$ is called unordered (also called strongly balanced in \cite{San09}) if $x\rightharpoondown y$ for any two disctinct $x,y\in W$.

A semiflow on $X$ is a continuous map $\Phi:\mathbb{R}^+\times X\to X$ which satisfies: (i) $\Phi_0=id_X$; (ii) $\Phi_t\circ\Phi_s=\Phi_{t+s}$ for $t,s\ge 0$. Here $\Phi_t(x)=\Phi(t,x)$ for $t\ge 0$ and $x\in X$ and $id_X$ is the identity map on $X$.

\begin{defn}\label{D:stongly-monotne}
 A continuous semi-flow $\Phi_t$ is called {\it monotone with respect to a $k$-solid cone $C$} if
$$\Phi_t(x)\thicksim\Phi_t(y),\, \text{ whenever } x\thicksim y \text{ and } t\ge 0;$$ and $\Phi_t$ is called {\it strongly monotone with respect to $C$} if $\Phi_t$ is monotone with respect to $C$ and
$$\Phi_t(x)\approx \Phi_t(y),\, \text{ whenever } x\ne y, x\thicksim y \text{ and } t>0.$$
\end{defn}

Let $x\in X$,  the {\it positive semi-orbit} of $x$ is denoted by $O^{+}(x)=\{\Phi_t(x):t\ge 0\}$.
A {\it negative semi-orbit} (resp. {\it full-orbit}) of $x$ is a continuous function $\psi:\mathbb{R}^-=\{t\in \mathbb{R}|t\le 0\}\to X$ (resp. $\psi:\mathbb{R}\to X$) such that $\psi(0)=x$ and, for any $s\le 0$ (resp. $s\in \mathbb{R}$), $\Phi_t(\psi(s))=\psi(t+s)$ holds for $0\le t\le -s$ (resp. $0\le t$). Clearly, if $\psi$ is a negative semi-orbit of $x$, then $\psi$ can be extended to a full orbit
\begin{equation}\label{E:full-orbit}
\tilde{\psi}(t)=\left\{\begin{split} & \psi(t), \,\,\,\,\,t\le 0,\\ &\Phi_t(x),\,\, t\ge 0.\end{split}\right.
\end{equation}
On the other hand, any full orbit of $x$ when restricted to $\mathbb{R}^-$ is a  negative semi-orbit of $x$. Since $\Phi_t$ is just a semiflow, a negative semi-orbit of $x$ may not exist, and it is not necessary to be unique even if one exists. Hereafter, we denote by $O^{-}_b(x)$ (resp. $O_b(x)$) a
negative semi-orbit (resp. full-orbit) of $x$. In particular, we write them as $O^{-}(x)$ (resp. $O(x)$) if such a negative semi-orbit (resp. full-orbit) is unique.

\indent {\it An equilibrium} is a point $x$ for which $O^{+}(x)=\{x\}$ (also called {\it a trivial semi-orbit}). Let $E$ be the set of all the equilibria of $\Phi_t$. A nontrivial $O^{+}(x)$ is said to be a {\it$T$-periodic orbit} for some $T>0$ if $\Phi_T(x)=x$.  A subset $S\subset X$ is called {\it positively invariant} if $\Phi_{t}(S)\subseteq S$ for any $t\geq 0$, and is called {\it invariant} if $\Phi_{t}(S)=S$ for any $t\geq 0$. Clearly, for any $x\in S$, there exists a negative semi-orbit of $x$ provided that $S$ is invariant.

The {\it$\omega$-limit set} $\omega(x)$ of $x\in X$ is defined by $\omega(x)=\cap_{s\ge 0}\overline{\cup_{t\ge s}\Phi_t(x)}$. If $O^{+}(x)$ is precompact (i.e., $\overline{O^{+}(x)}$ is compact), then $\omega(x)$ is nonempty, compact, connected and invariant. Hence, any $z\in \omega(x)$ admits in $\o(x)$ a negative semi-orbit, as well as a full-orbit.

Given a negative semi-orbit $O^{-}_b(x)$ of $x$, if $O^{-}_b(x)$ is precompact, then the {\it $\alpha$-limit set} $\a_b(x)$ of $O^{-}_b(x)$ is defined by $\alpha_{b}(x)=\{\cap_{s\ge 0}\overline{\cup_{t\ge s}\psi(-t)}\}$. We write $\a(x)$ as the $\alpha$-limit set of $O^{-}(x)$ if $x$ admits a unique negative semi-orbit.

A point $x\in X$  is called {\it non-wandering} if, for any neighborhood $U$ of $x$ and time $T>0$,  there is a $y\in U$ and a $t>T$ such that $\Phi_{t}(y)\in U$. Let $S\subset X$ be an invariant subset and $x_{1},x_{2}\in S$. For $\varepsilon,r>0$, a finite sequence $\{x_{1}=y_{1},y_{2}, \ldots, y_{n+1}=x_{2};\,t_{1},t_{2},\ldots,t_{n}\}$ of points $y_{i}\in S$ and times $t_{i}\geq r$ is called {\it $(\varepsilon,r)$-chain from $x_{1}$ to $x_{2}$ in $S$} if $\Phi_{t_{i}}(y_{i})\in B_{\varepsilon}(y_{i+1})$, $i=1,2\ldots,n$. Here $B_{\varepsilon}(y_{i+1})=\{z\in X:\norm{z-y_{i+1}}< \e\}$. A point $x\in S$ is called {\it chain-recurrent in $S$} if there is an $(\varepsilon,r)$-chain from $x$ to itself in $S$ for any $\varepsilon,r>0$. A subset $M\subset S$ is said to be chain recurrent if each point in $M$ is chain-recurrent in $S$. It is easy to see that the $\o$-limit set $\omega(x)$ of a precompact positive semi-orbit $O^{+}(x)$ is always chain-recurrent in both $\omega(x)$ and $X$ (See \cite[Appendix. Chain Recurrence]{Smi95}).

An invariant set $S\subset X$ is called {\it minimal} if it does not contain any proper invariant subset. Let $S\subset X$ be a minimal set. Then any $x\in S$ satisfies the following property: for any neighborhood $U$ of $x$, the set $N(x,U)=\{t>0: \Phi_{t}(x)\in U \}$ is relatively dense in $\R^+$, which means that there exists a positive number $l>0$, depending on $x$ and $U$, such that $N(x,U)\cap [t,t+l]\ne \emptyset$ for any $t\geq0 $.

\vskip 2mm

\begin{defn}\label{D:two-type-orbit}
A nontrivial semi-orbit $O^{+}(x)$ is called {\it pseudo-ordered} (also called {\it of Type-I}), if there exist two distinct point $\Phi_{t_1}(x),\Phi_{t_2}(x)$ in $O^{+}(x)$ such that $\Phi_{t_1}(x)\thicksim\Phi_{t_2}(x)$. Otherwise,
 $O^{+}(x)$ is called {\it unordered} (also called {\it of Type-II}), that is, any two distinct point $\Phi_{t_1}(x),\Phi_{t_2}(x)$ satisfies $\Phi_{t_1}(x)\rightharpoondown\Phi_{t_2}(x)$.\\
\end{defn}
\begin{defn}
Let $S\subset X$ be an invariant set with respect to $\{\Phi_t\}_{t\ge 0}$. $\Phi_t$ is said to {\it admit a flow extension} on $S$, if there is a flow $\{\tilde{\Phi}_t\}_{t\in \R}$ such that $\tilde{\Phi}_t(x)=\Phi_t(x)$ for any $x\in S$ and $t\ge 0.$
\end{defn}

If $S$ is locally compact, then $S$ admits a flow extension if and only if the negative semi-orbit (and hence the full-orbit) of any $x\in S$ is unique (see \cite[p.26, Theorem 2.3]{SY}).

\vskip 3mm

 Throughout this paper, we always impose the following assumptions:
\begin{itemize}
\item[{\bf (A1)}] $\Phi_t$ is a continuous semiflow which is strongly monotone with respect to a $k$-solid cone $C$.

\item[{\bf (A2)}] Any nonempty $\o$-limit set for $\Phi_t$ admits a flow extension.
\end{itemize}

\vskip 2mm

Now we are ready to present the main results of this paper. Among them, Theorems A-C mainly focus on the pseudo-ordered (Type-I) positive semi-orbits; while Theorem D concerns with the unordered (Type-II) positive semi-orbits. The proof of these theorems will be provided in the remaining sections.

\vskip 3mm
\noindent \textbf{Theorem A. } {\it Assume {\rm (A1)-(A2)} hold. Let $O^{+}(x)$ be a nontrivial pseudo-ordered precompact semi-orbit. Then the closure of any full-orbit in $\omega(x)$ is ordered.}
\vskip 2mm

 \begin{rmk}\label{Remark-on-ThmA}  \textnormal{ For $C^1$-smooth flows on $X=\mathbb{R}^n$, Theorem A was proved in \cite{San09} by using the Closing Lemma (see, e.g. \cite{Ar}). However, the Closing Lemma is still unknown in the infinite-dimensional spaces, and moreover; even in the finite dimensional case of $X=\mathbb{R}^n$, the Closing Lemma still cannot work because there is no smoothness assumption on $\Phi_t$ in our setting. Consequently, Theorem A is new even for the finite-dimensional case, which also gives an affirmative answer to the question posed in \cite[Remark 3]{San09}.}
 \end{rmk}

\vskip 3mm
\begin{rmk} \label{Remark-on-ThmB}  \textnormal{ As for the ordering property of $\omega(x)$ itself, it remains unknown. The following Theorem (Theorem B) will provide a trichotomy for the ordering property of $\omega(x)$. An immediate consequence of such trichotomy is that if $\omega(x)$ does not contain any equilibrium, then $\omega(x)$ itself is ordered. This partially solves the problem posed in \cite[Line 15-16, p.1984]{San09}.}
\end{rmk}

\vskip 3mm
\noindent \textbf{Theorem B.} {\it Assume {\rm (A1)-(A2)} hold. Let $O^{+}(x)$ be a nontrivial pseudo-ordered precompact semi-orbit. Then one of the following three alternatives must hold:
\begin{itemize}
\item[{\rm (i)}] $\omega(x)$ is ordered;

\item[{\rm (ii)}] $\omega(x)$ is unordered, and moreover, $\omega(x)\subset E$;

\item[{\rm (iii)}]  $\omega(x)$ possesses an ordered homoclinic property, i.e., there is an ordered and invariant subset $\tilde{B}\subsetneq\Omega$ such that, for any $p\in \Omega\setminus \tilde{B}$, it holds that
    $$\alpha(p)\cup\omega(p)\subset \tilde{B}\,\, \text{ and }\,\,\alpha(p)\subset E.$$\\
\end{itemize}

\vskip -10mm
In particular, if $\omega(x)\cap E=\emptyset$, then $\omega(x)$ itself is ordered. Moreover, if $C$ is complemented, then $\omega(x)$ is topologically conjugate to
a compact invariant set of a Lipschitz-continuous vector field in $\mathbb{R}^k$.
}

\vskip 2mm
Based on Theorem B, we can obtain the following

\vskip 3mm
\noindent \textbf{Theorem C.} {\it Assume {\rm (A1)-(A2)} hold. Assume that $k=2$ and $C$ is complemented. Let $O^{+}(x)$ be a nontrivial pseudo-ordered precompact semi-orbit. If $\omega(x)\cap E=\emptyset$, then $\omega(x)$ is a periodic orbit.}

\vskip 2mm
\begin{rmk}\label{Remark-on-ThmC}  \textnormal{When $\dim X<\infty$ ($X=\mathbb{R}^n$), Theorem C was proved by Sanchez \cite[Theorem 1]{San09}, which we now {\it refer to the Poincar\'{e}-Bendixson Theorem for $\Phi_t$.} The crucial tools of the proof in \cite{San09} are the generalized Perron-Frobenius Theorem (see [22]) and theory of invariant manifolds in $\mathbb{R}^n$, which strongly depend on the $C^1$-smoothness assumption of $\Phi_t$. Here, proved the Poincar\'{e}-Bendixson Theorem for $\Phi_t$ on a Banach space $X$ without the smoothness assumption.}
\end{rmk}

\vskip 3mm
\noindent \textbf{Theorem D.} {\it Assume {\rm (A1)-(A2)} hold. Let $O^{+}(x)$ be a nontrivial unordered precompact semi-orbit for $\Phi_t$. Then $\omega(x)$ itself is either an equilibrium or unordered.}
\vskip 3mm

\begin{rmk} \label{Remark-on-ThmD}  \textnormal{Finally, it deserves to point out that we here only require the flow extension on the $\o$-limit sets (see (A2)). In a semiflow which is generated by a differential equation such as parabolic equation or a certain delayed equation (see, e.g. \cite[Part III]{SY}), the semiflow itself does not admit a flow extension in general, but an $\o$-limit set does. Further applications of the theory developed in these theorems will be provided in a forthcoming paper.}
\end{rmk}

\section{Proof of Theorem A}

We first need the following Proposition:

\begin{prop}\label{P:non-wondering}
If x is a non-wandering point and there is some $T>0$ such that $x\sim\Phi_{T}(x)$ with $x \neq\Phi_{T}(x)$, then $\overline{O^{+}(x)}$ is ordered.
\end{prop}
\begin{proof}
 Define $T_{1}=\sup\{t\geq T:x\thicksim\Phi_s(x)$ for all $s\in[T,t]\}$. Suppose that $T_{1}<+\infty$. Then
 $x\thicksim \Phi_{T_{1}}(x)$. If $x=\Phi_{T_{1}}(x)$, we have done by using the same argument in \cite[Proposition 2]{San09}. If $x\ne\Phi_{T_{1}}(x)$, then fix a $\tau>0$, it follows from the strongly monotone property and the continuity of $\Phi_\tau$ that one can find  a neighborhood $U$ of $x$ and a neighborhood $V$ of $\Phi_{T_{1}}(x)$ such that $\Phi_\tau(U)\thicksim\Phi_\tau(V)$, and hence $\Phi_t(U)\thicksim\Phi_t(V)$ for any $t\ge \tau$.

Choose an $\varepsilon>0$ so small that $B_{\varepsilon}(\Phi_{T_{1}}(x))\subset V$, where $B_{\varepsilon}(x)$ is the ball centered at $x\in X$ with the radius $\varepsilon>0$. For such $\varepsilon>0$, choose some $\delta_1>0$ such that $\Phi_{T_1}(B_{\delta_1}(x))\subset B_{\varepsilon/3}(\Phi_{T_{1}}(x))$, and hence, there exists some $l>0$ such that $\Phi_{T_{1}+s}(x)\in B_{\varepsilon/3}(\Phi_{T_{1}}(x))$ for any $s\in[-l,l]$. So, for every $s\in [-l,l]$, one can find some $\delta(s)\in (0,\delta_1)$ such that $\Phi_{T_{1}+s}(B_{\delta(s)}(x))\subset B_{\varepsilon}(\Phi_{T_{1}}(x))$. Moreover, it follows from the compactness of $[-l,l]$ that $\inf_{s\in [-l,l]}\delta(s)=\delta>0$. As a consequence,
\begin{equation}\label{s-contain}
\Phi_{T_{1}+s}(B_{\delta}(x))\subset B_{\varepsilon}(\Phi_{T_{1}}(x))
\, \text{ for all }\,s\in [-l,l].
\end{equation}
By choosing such $\delta>0$ small, if necessary, we assume that $B_\delta(x)\subset U$. Recall that  $B_{\varepsilon}(\Phi_{T_{1}}(x))\subset V$. Then
\begin{equation}\label{delta-T1-order}
\Phi_t(B_\delta(x))\thicksim\Phi_t(B_{\varepsilon}(\Phi_{T_{1}}(x)))\, \text{ for all }\,t\geq \tau.
\end{equation}
Since $x$ is a non-wondering point, there are two sequences $y_n\to x$ and $\tau_n\to \infty$ such that
$\Phi_{\tau_n}(y_n)\to x$ as $n\to \infty$. Therefore, one may assume without loss of generality that $y_n\in B_\delta(x)$ and $\tau_n>\tau$ for any $n\ge 1$. By virtue of \eqref{s-contain}, it entails that
\begin{equation*}\label{y-n-in}
\Phi_{T_{1}+s}(y_n)\subset B_{\varepsilon}(\Phi_{T_{1}}(x)),
\, \text{ for all }\,s\in [-l,l]\text{ and }n\ge 1.
\end{equation*}
Together with \eqref{delta-T1-order}, this implies that
\begin{equation*}
\Phi_{\tau_n}(y_n)\thicksim\Phi_{\tau_n}(\Phi_{T_{1}+s}(y_n)),
\, \text{ for all }\,s\in [-l,l]\text{ and }n\ge 1.
\end{equation*}
By letting $n\to \infty$, it yields that $x\sim \Phi_{T_{1}+s}(x)$ for all $s\in [-l,l]$, which contradicts the definition of $T_{1}$. Thus, $T_{1}=+\infty$.

Similarly, we define $T_{2}=\inf\{t>0:x\thicksim\Phi_t(x)\}$. By following the same argument for proving $T_{1}=+\infty$, one can obtain that $T_{2}=0$. Thus, we have proved that $O^{+}(x)$ is ordered. Moreover, by the closedness of $C$, we can further obtain that  $\overline{O^{+}(x)}$ is ordered.
\end{proof}

\begin{rmk}\label{R:stronly-com}
\textnormal{ By (A1), a simple fact can be deduced from the above proof is that: If $x\thicksim \Phi_T(x)$ with $x\neq\Phi_T(x)$, then for any fixed $\tau>0$, there exists a neighborhood $U$ of $x$ and a neighborhoods $V$ of $\Phi_{T}(x)$ such that $\Phi_t(U)\thicksim\Phi_t(V)$ for all $t\geq \tau$.}
\end{rmk}
\vskip 2mm

Now we are ready to prove Thoerem A.
\vskip 2mm
%\begin{lem}\label{order-orbit}
% Let $O^{+}(x)$ be a precompact positive semi-orbit of Type-I. Then the closure of any full-orbit in $\omega(x)$ is ordered.
%\end{lem}

{\it Proof of Theorem A.} By virtue of (A2), any $p\in \o(x)$ admits a unique full-orbit $O(p)\subset \o(x)$. Without loss of generality, we assume that $p$ is not an equilibrium.

Denote by $\psi:\R\to \o(x)$ the full-orbit with $\psi(0)=p$ and $\psi(t)=\Phi_t(p)$ for $t\ge 0$. So
$\psi(s)$ is not an equilibrium for each $s\in \R$.  we now {\it assert that there exists some $T>0$ such that $\psi(s) \neq \Phi_{T}(\psi(s))$ and $\psi(s) \thicksim\Phi_{T}(\psi(s))$}. In fact, since $O^+(x)$ is nontrivial and pseudo-ordered, there exist $t_2>t_1\ge 0$ such that $\Phi_{t_1}(x)\thicksim \Phi_{t_2}(x)$ and $\Phi_{t_1}(x)\neq\Phi_{t_2}(x)$. Choose a sequence $\tau_k\rightarrow +\infty$ such that $\Phi_{\tau_{k}}(\Phi_{t_1}(x))\rightarrow \psi(s)$. Then $\Phi_{\tau_{k}}(\Phi_{t_2}(x))\rightarrow \psi(s+t_2-t_1)$. If $\psi(s)\neq \psi(s+t_2-t_1)$, then let $T=t_1-t_2$, and hence,  $\psi(s)\thicksim \psi(s+T)$ with $\psi(s)\neq\psi(s+T)$. Thus, we have done.
If $\psi(s)=\psi(s+t_2-t_1)$, then fix some $T_{1}>0$ and choose a neighborhood $U$ (resp. $V$) of $\Phi_{T_{1}}(\Phi_{t_1}(x))$ (resp. $\Phi_{T_{1}}(\Phi_{t_2}(x))$) such that $\Phi_t(U)\thicksim\Phi_t(V)$ for $t\geq 0$. By the continuity of $\Phi_t$, we take $0<\delta<T_{1}$ such that $\{\Phi_{T_{1}+\t}(\Phi_{t_1}(x)):\t\in[-\delta,\delta]\}\subset U$ and $\{\Phi_{T_{1}+\t}(\Phi_{t_2}(x)):\t\in[-\delta,\delta]\}\subset V$. Therefore,
$\Phi_{\tau_{k}+s_{1}}(\Phi_{t_1}(x))\thicksim\Phi_{\tau_{k}+s_{2}}(\Phi_{t_2}(x))$ for all $s_{1},s_{2}\in[-\delta,\delta]$ and $k$ sufficiently large. Together with $\psi(s)=\psi(s+t_2-t_1)$, this implies that the set $\{\psi(s+\t): \t\in[-\delta,\delta]\}$ is ordered. Recalling that $\psi(s)$ is not an equilibrium, we can find some $T\in [0,\delta]$ such that $\psi(s) \neq \Phi_{T}(\psi(s))$ and $\psi(s) \thicksim\Phi_{T}(\psi(s))$. Thus, we have proved the assertion.

Recall that any point in $\o(x)$ is non-wandering. Together with the assertion, Proposition \ref{P:non-wondering} implies that $\overline{O^+(\psi(s))}$ is ordered for each $s\in \R$.

For any $z,w\in \overline{O(p)}$, we only consider the case that there exist $s_n,t_n\to -\infty$ (as $n\to \infty$) such that $\psi(s_n)\to z$ and $\psi(t_n)\to w$. Other cases are similar. Without loss of generality, we assume that $s_n<t_n$ for each $n$. So $\psi(t_n)\in O^+(\psi(s_n))$, and hence, one has $\psi(t_n)\thicksim \psi(s_n)$ for each $n$. Because of the closeness of $C$, it yields that $z\thicksim w$. Thus, we have obtained $\overline{O(p)}$ is ordered.
$\qquad\qquad\qquad\qquad\qquad\qquad\qquad\qquad\qquad\square$

 \section{Proofs of Theorems B and D}

In this section, we focus on the ordering property of the $\omega(x)$ itself under the fundamental assumptions (A1)-(A2). Theorems B reveals the information of $\o(x)$ with nontrivial pseudo-ordered semi-orbits; while Theorem D concerns with nontrivial unordered semi-orbits.

We first consider the omega limit sets of nontrivial pseudo-ordered precompact semi-orbits. By virtue of Theorem A,  the closure of any full-orbit in $\omega(x)$ is ordered. As a consequence, one may assume without loss of generality that any full-orbit $O(a)$  in $\omega(x)$ satisfies
\begin{equation}\label{funda-assum}
\begin{split}
&\textnormal{{\bf (F)}}\quad\quad\quad\quad \quad\quad\quad \overline{O(a)}\subsetneq\omega(x),\,\, \textnormal{ for any } a\in \omega(x).\quad\quad\quad\quad \quad\quad\quad \
\end{split}
\end{equation}

By virtue of (A2), the negative semi-orbit (hence the full-orbit) of any $a\in \o(x)$ is unique. For the sake of convenience, we hereafter write the unique negative semi-orbit of $a\in \o(x)$ as $\{\Phi_{-s}(a)|s\ge 0\}$ satisfying
\begin{equation}\label{E:full-orbit-of-a}
\Phi_t(\Phi_{-s}(a))=\Phi_{t-s}(a),\,\,\,\text{ for any }t,s\ge 0.
\end{equation}

In order to prove Theorem B, we first present the following three technical lemmas.

\begin{lem}\label{cmpset-ordered}
 Let two compact sets $K_{1},K_{2}$ satisfy $K_{1}\cap K_{2}=\emptyset$ and $K_{1}\thicksim K_{2}$. Then there are open sets $U\supset K_{1}, V\supset K_{2}$ and a $T>0$ such that $\Phi_t(U)\thicksim \Phi_t(V)$ for all $t>T$.
\end{lem}
\begin{proof}
Since $K_{1}\cap K_{2}=\emptyset$, any $x\in K_{1}$ and $y\in K_{2}$ satisfy $x\neq y$. By Remark \ref{R:stronly-com}, there exist open sets $U_{x,y}\ni x$, $V_{x,y}\ni y$ and time $T_{x,y}>0$ such that $\Phi_t(U_{x,y})\thicksim \Phi_t(V_{x,y})$ for all $t>T_{x,y}$.

Note that the family $\{U_{x,y}\}_{x\in K_{1}}$ of open sets covers of $K_{1}$. Then we may choose a finite subcover $U_y:=\cup_{i=1}^{m}U_{x_{i},y}$ of $K_1$. Denote $V_{y}=\bigcap_{i=1}^{m}V_{x_{i},y}$ and $T_{y}=\max\{T_{x_{i},y}:i=1,2\ldots m\}$, then $\Phi_t(U_{y})\thicksim \Phi_t(V_{y})$ for all $t\geq T_{y}$.

Moreover, since $\{V_{y}\}_{y\in K_{2}}$ is an open cover of $K_{2}$, so one may choose a finite subcover $V=\bigcup_{j=1}^{n}V_{y_{j}}$ of $K_2$. Denote $U=\bigcap_{j=1}^{n}U_{y_{j}}$ and $T=\max\{T_{y_{j}}:j=1,2\ldots n\}$, then $\Phi_t(U)\thicksim \Phi_t(V)$ for all $t\geq T$.
\end{proof}

\begin{lem}\label{parti-ordered}
Let $A \subset \omega(x)$ be an invariant compact set. If there exists some $a\in \omega(x)\setminus A$ with $a\thicksim A$, then one has $A\thicksim \omega(x)$.
\end{lem}

\begin{proof}
By Lemma \ref{cmpset-ordered}, there exist open sets $U\ni a$,$V\supseteq A$ and $T_{0}>0$ such that $\Phi_t(U)\thicksim \Phi_t(V)$ for all $t\geq T_{0}$. Since A is invariant, one has $\Phi_t(U)\thicksim A$ for all $t\geq T_{0}$. Noticing $a\in\omega(x)$,  one can find a $T_{1}>0$ such that $\Phi_{T_1}(x)\in U$. So we obtain $\Phi_{T_{1}+t}(x)\thicksim A$ for $t>T_{0}$, which directly implies that $A\thicksim \omega(x)$.
\end{proof}

\begin{lem}\label{co-limit}
 %\textit{\textbf{(co-limiting principle)}}
 \begin{description}
\item{{\rm (i)}} If $a_{1}\thicksim a_{2}$ and there is a sequence $\tau_{k}\to \infty$ such that $\Phi_{\tau_k}(a_{1})\rightarrow c$ and $\Phi_{\tau_k}(a_{2})\rightarrow c$, then either $c\in E$ or $\overline{O(c)}$ is ordered.

\item{{\rm (ii)}} Let $a_1,a_2\in \o(x).$ If $a_1\rightharpoondown a_{2}$ and there is a sequence $\tau_{k}\to \infty$ such that $\Phi_{-\tau_k}(a_{1})\rightarrow c$ and $\Phi_{-\tau_k}(a_{2})\rightarrow c$, then either $c\in E$ or $\{\Phi_t(c)\mid t\in [-\d,\d]\}$ unordered for some $\d>0$.
    \end{description}
\end{lem}

\begin{proof}
(i) Fix $T_{1}>0$ and choose open sets $U\ni \Phi_{T_{1}}(a_{1})$, $V\ni \Phi_{T_{1}}(a_{2})$ and such that $\Phi_t(U)\thicksim\Phi_t(V)$ for $t\geq 0$. By continuity of $\Phi_t$, take $T_{1}>\delta>0$ such that $\{\Phi_{T_{1}+s}(a_{1}):s\in[-\delta,\delta]\}\subset U$ and $\{\Phi_{T_{1}+s}(a_{2}):s\in[-\delta,\delta]\}\subset V$. Then we have $\Phi_{\tau_{k}+s_{1}}(a_{1})\thicksim\Phi_{\tau_{k}+s_{2}}(a_{2})$ for $s_{1},s_{2}\in[-\delta,\delta]$ and all $k$ sufficiently large. Recall that $c\in \o(a_i), i=1,2$. Then (A2) implies that the full-orbit $O(c)$ of $c$ is well-defined and satisfies $O(c)\subset \o(a_i),i=1,2.$ So, by letting $k\to \infty$, we obtain that $\Phi_{s_{1}}(c)\thicksim\Phi_{s_{2}}(c)$ for
any $s_{1},s_{2}\in [-\delta,\delta]$. This implies that $O^+(a_{i}),i=1,2,$ are of type-I. Suppose that $c\notin E$. Then, by Theorem A and $\overline{O(c)}\subset\omega(a_{1})$, we directly obtain that $\overline{O(c)}$ is ordered.

 (ii) Since $a_1,a_2\in \o(x)$, $\overline{O(a_i)}, i=1,2,$ are well-defined.  Noticing that $a_{1}\rightharpoondown a_{2}$, there exist a neighborhood $U_1$ (resp. $U_2$) of $a_1$ (resp. $a_2$) such that $U_1 \rightharpoondown U_2$.  Take a $\delta>0$ such that $\{\Phi_s(\Phi_{-\delta}(a_{i})): s\in[0,2\delta]\}\subset U_i$ for $i=1,2$. As a consequence, $\Phi_{-\tau_{k}-\delta+s_{1}}(a_{1})\rightharpoondown\Phi_{-\tau_{k}-\delta+s_{2}}(a_{2})$ for any $s_{1},s_{2}\in[0,2\delta]$ and all $k>0$ sufficiently large. Suppose now that $c\notin E$. Then we assert that $\Phi_{s_{1}}(c)\rightharpoondown\Phi_{s_{2}}(c)$ for any $s_{1}\neq s_{2}\in [-\d,\delta]$. Otherwise, one can find some $s_1,s_2\in [-\d,\delta]$ such that $\Phi_{-\tau_{k}+s_{1}}(a_{1})\sim\Phi_{-\tau_{k}+s_{2}}(a_{2})$ for $k$ large enough, a contradiction.
\end{proof}

By the assumption {\bf (F)} in \eqref{funda-assum},  we may classify the closure $\overline{O(a)}$ of any given orbit
in $\omega(x)$ into the following two types:

${\rm (P1)}$: $\quad \overline{O(a)}\thicksim \omega(x)$;  otherwise,

${\rm (P2)}$:  $\quad $ there is some $z\in \omega(x)\setminus \overline{O(a)}$ such that $z\rightharpoondown y$ for some $y\in \overline{O(a)}$.

\vskip 2mm

\noindent By Lemma \ref{parti-ordered} and the invariance of $\overline{O(a)}$, it is easy to see that $\overline{O(a)}$ satisfies (P2) if and only if,
for any $z\in \omega(x)\setminus \overline{O(a)}$, there exists some $t_0\in \mathbb{R}$ such that $z\rightharpoondown \Phi_{t_0}a$.

Define $$B=\{\overline{O(b)}\subset \omega(x):\overline{O(b)}\textnormal{ satifies  (P1)} \},\,\, \textnormal{  }\,\, \,\,\tilde{B}=\bigcup_{\overline{O(b)}\in B}\overline{O(b)},$$ and
$$A=\{\overline{O(a)}\subset \omega(x):\overline{O(a)}\textnormal{ satifies  (P2)} \}\,\, \textnormal{ and  }\,\, \,\,\tilde{A}=\bigcup_{\overline{O(a)}\in A}\overline{O(a)}.$$
Clearly, $\tilde{A}\cup \tilde{B}=\omega(x)$, $A\cap B=\emptyset$ and $\tilde{B}\thicksim \omega(x)$ (hence $\tilde{B}$ is ordered). Moreover, we have further properties for these sets:

\begin{prop}\label{property-of-A-B}
 {\rm (i)} $\overline{O(a)}\in A$ if and only if $\tilde{B}\subsetneq\overline{O(a)}$. Moreover, if  $A\neq\emptyset$ then $\tilde{A}\cup \tilde{B}=\tilde{A}=\omega(x)$.

{\rm (ii)} Let $\overline{O(a_i)}\in A, i=1,2$ satisfy $\overline{O(a_1)}\ne \overline{O(a_2)}$. Then
$\overline{O(a_1)}\nsubseteq \overline{O(a_2)}$ and $\overline{O(a_2)}\nsubseteq \overline{O(a_1)}$. Moreover, $\overline{O(a_1)}\cap \overline{O(a_2)}\supseteq\tilde{B}.$
\end{prop}

%\begin{rmk} {\red Here, $\overline{O(a_i)}\in A, i=1,2$ are just two different branches of orbit in $\omega(x)$. If $O^{+}(a_1)\subseteq O^{+}(a_2)$ or $O^{+}(a_2)\subseteq O^{+}(a_1)$, then $\overline{O(a_i)}\in A, i=1,2$ are two different branches of orbit respect to one same point. If not, they are different branches respect to the different points.}
%\end{rmk}

\begin{proof}
(i) If $B=\emptyset$, We have done. It only needs to consider the case that $B\neq\emptyset$.

{\it Necessity.} Let $\overline{O(a)}\in A$, we first show $\tilde{B}\subset\overline{O(a)}$. Suppose that there is some
$b\in \tilde{B}\setminus\overline{O(a)}$. Then $b\sim \omega(x)$, and hence $b\sim \overline{O(a)}$. It then follows from Lemma \ref{parti-ordered} that $\overline{O(a)}\sim \omega(x)$, which implies that $\overline{O(a)}\in B$, a contradiction to $A\cap B=\emptyset.$ So we have obtained that  $\tilde{B}\subset\overline{O(a)}$. Suppose that
$\tilde{B}=\overline{O(a)}$. Then $\overline{O(a)}(=\tilde{B})\sim \omega(x),$ and again, one has $\overline{O(a)}\in B$, a contradiction. Thus one obtains that $\tilde{B}\subsetneq\overline{O(a)}$.

{\it Sufficiency.} Suppose that $\overline{O(a)}\in B$. Then one has $\tilde{B}\subsetneq \overline{O(a)}
 \subset \tilde{B}$, a contradiction.  Moreover, if $A\neq\emptyset$, then $\tilde{B}\subsetneq \overline{O(a)}$ for some $\overline{O(a)}\in A$. Note that $\overline{O(a)}\subset \tilde{A}$. Then one has $\tilde{A}\cup \tilde{B}=\tilde{A}=\omega(x)$.

 (ii) Suppose that $\overline{O(a_{1})}\subseteq \overline{O(a_{2})}$. Then $\overline{O(a_{1})}\subsetneq\overline{O(a_{2})}$, because $\overline{O(a_{1})}\neq\overline{O(a_{2})}$.
 As a consequence, there exists $b\in\overline{O(a_{2})}$ such that $b\notin\overline{O(a_{1})}$. Recall that $\overline{O(a_{2})}$ is ordered. Then $b\thicksim\overline{O(a_{1})}$. By Lemma \ref{parti-ordered}, one has $\overline{O(a_{1})}\thicksim \omega(x)$, contradicting that $\overline{O(a_{1})}\in A$. Thus, we have proved that $\overline{O(a_1)}\nsubseteq \overline{O(a_2)}$.  Similarly, we can also obtain $\overline{O(a_2)}\nsubseteq \overline{O(a_1)}$.

 The fact $\overline{O(a_1)}\cap \overline{O(a_2)}\supseteq\tilde{B}$ is directly from (i).
\end{proof}

\begin{prop}\label{B-nonempty}
Let $O^{+}(x)$ be a nontrivial pseudo-ordered precompact semi-orbit and assume that {\bf {\rm (F)}} in \eqref{funda-assum} holds. Assume also that $B\ne \emptyset$. Then one of the following alternatives holds:

{\rm (i)} $\omega(x)$ is ordered; or otherwise,

{\rm (ii)} $A\ne \emptyset$ and $\omega(x)$ possesses the following ordered homoclinic property: Given any $\overline{O(a)}\in A$, it holds that $\alpha(a)\cup\omega(a)\subset \tilde{B}$; and moreover, $\alpha(a)\subset E$.
\end{prop}
\begin{proof}
\indent  If $A=\emptyset$ then $\omega(x)$ is clearly ordered. Suppose that the cardinality of the set $A$ is equal to $1$, then it follows from Propsition \ref{property-of-A-B}(i) that $\omega(x)=\tilde{A}=\overline{O(a)}$ for some $a\in \omega(x)$, which contradicts (F).

 So, it suffices to consider the case that the cardinality of $A$ is at least $2$. Hence, for any $\overline{O(a_{1})}\subset A$, there exists $\overline{O(a_{2})}\subset A$ such that
$\overline{O(a_{1})}\ne \overline{O(a_{2})}$. Moreover, one may assume without loss of generality that $a_1\rightharpoondown a_2$. Indeed, it follows from the argument following the definition of (P2) that there is a $y\in \overline{O(a_{2})}\setminus \overline{O(a_{1})}$ such that $y\rightharpoondown \Phi_{t_0}(a_1)$ for some $t_0 \in\R$. Moreover, we can find $t_1\in R$ such that $\Phi_{t_1}(a_2)\rightharpoondown \Phi_{t_0}(a_1)$ since $y\in \overline{O(a_2)}$. Note that $\overline{O(a_{1})}=\overline{O(\Phi_{t_0}(a_1))}$ and  $\overline{O(a_{2})}=\overline{O(\Phi_{t_1}(a_2))}$. Thus we may assume that $a_1\rightharpoondown a_2$.

Choose any sequence $t_k\to \infty$ such that $\Phi_{-t_k}(a_{i})\rightarrow c_{i}$, for $i=1,2.$ Then either (i) $c_{1}=c_{2}$; or otherwise, (ii) $c_{1}\ne c_{2}$.

{\it We claim that case {\rm (ii)} cannot happen.}
Before proving this claim, we first show how it implies our conclusion. In fact, for any $c\in \alpha(a_{1})$, there is a sequence $t_k\to \infty$ such that $\Phi_{-t_{k}}(a_{1})\rightarrow c$ as $t_k\to \infty$.
By virtue of the claim, one can choose a subsequence of $\{t_k\}$, still denoted by $\{t_k\}$, such that  $\Phi_{-t_{k}}(a_{2})\rightarrow c$ as $t_k\to \infty$. Consequently, $c\in \alpha(a_{2})$, and hence, $\alpha(a_{1})\subset \alpha(a_2)$. Similarly, one can get $\alpha(a_{2})\subset \alpha(a_1)$. Thus,  $\alpha(a_{1})=\alpha(a_2)$. By Lemma \ref{co-limit}(ii), we can further obtain that $\alpha(a_{1})=\alpha(a_2)\subset E$ (Otherwise, choose $z\in \alpha(a_{1})\setminus E$ and some $s_{k}\rightarrow +\infty$ such that $\Phi_{-s_{k}}(a_{1})\rightarrow z\in \a(a_1)=\a(a_2)\subset \o(x)$.  Again, by the claim, one can assume that $\Phi_{-s_{k}}(a_{2})\rightarrow z$.  Since (A2) holds, $\overline{O(z)}$ is well-defined. So Lemma \ref{co-limit}(ii) implies that $\overline{O(z)}$ is locally non-ordered. On the other hand, noticing $\overline{O(z)}\subset \omega(x)$, it follows from Theorem A that $\overline{O(z)}$ is ordered, a contradiction.)

Since $\alpha(a_{1})=\alpha(a_2)$, one has $\alpha(a_{1})\subsetneq \overline{O(a_{1})}$ (For otherwise, $\overline{O(a_{1})}=\alpha(a_{1})=\alpha(a_{2})\subseteq \overline{O(a_{2})}$, which contradicts Proposition \ref{property-of-A-B}(ii)).  So, by Lemma \ref{parti-ordered}, we obtain that $\alpha(a_{1})\thicksim \omega(x)$. As a consequence, $\alpha(a_{1})(=\alpha(a_2))\subset\tilde{B}\cap E$.

We now prove $\omega(a_1)\subset\tilde{B}$. Since $\alpha(a_2)(=\alpha(a_1))\subset\tilde{B}$ and $\overline{O(a_{1})}\ne \overline{O(a_{2})}$, one has $a_1\thicksim \alpha(a_2)$ and $a_1\notin \alpha(a_2)$. By Lemma \ref{cmpset-ordered}, one can find an open set $U\supset \alpha(a_2)$ and time $T>0$ such that $\Phi_t(U)\thicksim \Phi_t(a_1)$ for $t\ge T$. Choose $\tau>0$ so large that $\Phi_{-t}(a_2)\in U$ for any $t\ge \tau$. Then $\Phi_{-t}(a_2) \thicksim \Phi_{T}(a_1)$ for all $t\geqq \tau-T $. So, by the monotonicity of $\Phi_t$, we have  $\Phi_t(a_1)\thicksim a_2$ for all $t\geq \tau$. This implies that $\omega(a_1)\thicksim a_2$. However, note that $a_2 \notin \omega(a_1)$ (otherwise $\overline{O(a_{2})}\subseteq\omega(a_1)\subseteq \overline{O(a_{1})}$, a contradiction to Proposition \ref{property-of-A-B}(ii)). Again, by Lemma \ref{parti-ordered}, one has $\omega(a_1)\thicksim \omega(x)$, which implies that $\omega(a_1)\subset \tilde{B}$. Thus, we have obtain all the statements in Proposition \ref{B-nonempty}.

Finally, it suffices to prove the claim above.
To this end, suppose that $c_1\ne c_2$. Then we have $c_1\rightharpoondown c_2$, whose proof will be postponed to Lemma \ref{undd-1} below. So, for the full-orbit $\overline{O(a_i)}, i=1,2$, obtained above, it holds that $a_{i}\in\overline{O(c_{i})}$, $i=1,2$. (Otherwise, say $a_{1}\notin\overline{O(c_{1})}$. Recalling that
$\overline{O(c_{1})}\subset\overline{O(a_{1})}$ and $\overline{O(a_{1})}$ is ordered, we have  $a_{1}\sim\overline{O(c_{1})}$. Thus, Lemma \ref{parti-ordered} implies that $\omega(x)\sim \overline{O(c_{1})}$, a contradiction to  $c_1\rightharpoondown c_2$.) So, $\overline{O(a_i)}\subset \overline{O(c_i)}$. Note also that $c_i\in \a(a_i)$. Then we have  $\overline{O(a_{i})}=\a(a_{i})$ for $i=1,2$.

Now we can choose a sequence $\tau_k\to \infty$ such that
 $\Phi_{-\tau_{k}}(a_{1})\rightarrow a_{1}$ and  $\Phi_{-\tau_{k}}(a_{2})\rightarrow b$ as $\tau_{k}\rightarrow \infty$. Clearly,
 $b\neq a_{1}$ (Otherwise, $a_{1}\in \overline{O(a_{2})}$. By Theorem A, we know $a_{1}\thicksim a_{2}$, contracting to $a_{1}\rightharpoondown a_{2}$). Then, again by the forthcoming Lemma \ref{undd-1}, one obtains that $O(a_{1})\rightharpoondown b$ and $a_{1}\rightharpoondown O(b)$. Since $b\in \a(a_2)\subset \o(x),$ $\overline{O(b)}$ is well-defined. Obviously, $\overline{O(b)}\in A$. Now choose some $d\in \tilde{B}\subsetneq \overline{O(b)}\in A$ (because $B\neq\emptyset$). Clearly, $d\ne a_{1}$ and $d\thicksim a_{1}$. So, there exists a neighborhood $U$ (resp. $V$) of $a_{1}$ (resp. of $d$) and a $T_{1}> 0$ such that $\Phi_t(U)\sim \Phi_t(V)$ for $t\ge T_{1}$. Let $\{s_{l}\}_{l=1}^{\infty}\subset \mathbb{R}$ be a sequence so that
 $\Phi_{s_{l}}(b)\rightarrow d$. Then one has $\Phi_{T_1-\tau_k}(a_{1})\sim \Phi_{T_1+s_l}(b)$ for all $k,l$ sufficiently large.
 Now take $\tau_k$ so large that $T_{1}-\tau_k<0$, then it follows from the monotonicity of $\Phi_{t\ge0}$ that  $a_1\thicksim \Phi_{\tau_k+s_l}b$, which contradicts $a_{1}\rightharpoondown O(b)$. Thus, we have completed the proof of the claim.
\end{proof}

\begin{lem}\label{undd-1}
Let $a_1,a_2\in \o(x)$. Assume that $a_{1}\rightharpoondown a_{2}$ and $\Phi_{-t_{k}}(a_{i})\rightarrow b_{i}$ as $t_k\rightarrow +\infty$ for $i=1,2.$ If $b_{1}\neq b_{2}$, then $b_{1}\rightharpoondown b_{2}$. Furthermore, $b_{1}\rightharpoondown O(b_{2})$ and $O(b_{1})\rightharpoondown b_{2}$.
\end{lem}
\begin{proof}
By the monotonicity of $\Phi_t$, $\Phi_{-t_{k}}(a_{1})\rightharpoondown\Phi_{-t_{k}}(a_{2})$ for any $k>0$.
 Suppose that $b_{1}\thicksim b_{2}$. By Remark \ref{R:stronly-com}, there exists a neighborhood $U$ (resp. $V$) of $b_{1}$ (resp. $V$) and $T_{2}>0$ such that $\Phi_t(U)$$\thicksim$$\Phi_t(V)$ for $t\geq T_{2}$. Choose some $t_{k}>T_{2}$ such that $\Phi_{-t_{k}}(a_{1})\in U$ and $\Phi_{-t_{k}}(a_{2})\in V$. Then $\Phi_{T_{2}-t_{k}}(a_{1})\thicksim\Phi_{T_{2}-t_{k}}(a_{2})$, which implies that $a_{1}\thicksim a_{2}$, a contradiction. Thus, $b_{1}\rightharpoondown b_{2}$.

We now prove $O(b_{1})\rightharpoondown b_{2}$, the proof of $b_{1}\rightharpoondown O(b_{2})$ is similar. For any $c_1,c_2\in \o(x)$ with $c_1\rightharpoondown c_2$, define $$\Gamma(c_{1},c_{2})=\sup\{t\geq 0:\Phi_s(c_{1})\rightharpoondown c_{2}\,\textnormal{ for all }\,s\in[-t,t]\}.$$ Clearly, $\Gamma(c_{1},c_{2})> 0$. Define a positive function
$f(t)=\Gamma(\Phi_{-t}(a_{1}),\Phi_{-t}(a_{2}))$ for $t\ge 0$. By strongly monotone property of $\Phi_t$, it is easy to see that $f(t)$ is nondecreasing with respect to $t\ge 0$. Note also that $f(t_{k})\leq \Gamma(b_{1},b_{2})$ for all $t_k>0$. Then $\Gamma(b_{1},b_{2})$ is an upper bound of $f(t)$ on $[0,+\infty)$.

We assert that $\Gamma(b_{1},b_{2})=\sup_{t\ge 0} \{f(t)\}$. Let $\{l_n\}_{n=1}^{+\infty}$ be a increasing sequence with positive number such that $l_n\rightarrow \Gamma(b_{1},b_{2})$ as $n\rightarrow +\infty$. Clearly, the set
$L=\{\Phi_s(b_{1}):s\in [-l_n,l_n]\}$ is compact and $L\rightharpoondown b_2$.  Then there exists a neighborhood $U$ (resp. $V$) of $L$ (resp. $b_2$)  such that $U\rightharpoondown V$. Since $\Phi_{-t_{k}}(a_{2})\in V$ and $\Phi_s(\Phi_{-t_{k}}(a_{1}))\in U$ for all $s\in [-l_n,l_n]$ and all $k$ sufficiently large, one has $\Phi_{s-t_{k}}(a_{1})\rightharpoondown\Phi_{-t_{k}}(a_{2})$ for all $s\in  [-l_n,l_n]$ and $k$ sufficiently large. So, one can find some $\tau> 0$ such that $f(\tau)\ge  l_n$ for any $n\in N^+$. Thus, we have proved the assertion.

Finally, we prove that $\Gamma(b_{1},b_{2})= +\infty$ (and hence, $O(b_{1})\rightharpoondown b_{2}$). Suppose that
$\Gamma(b_{1},b_{2})<+\infty$. Then $\Phi_{-\Gamma(b_{1},b_{2})}(b_{1})\thicksim b_{2}$ or $\Phi_{\Gamma(b_{1},b_{2})}(b_{1})\thicksim b_2$.
 Without loss of the generality, we assume $\Phi_{\Gamma(b_{1},b_{2})}(b_{1})\thicksim b_{2}$. Then then exists a neighborhood $U$ (resp. $V$) of $\Phi_{\Gamma(b_{1},b_{2})}(b_{1})$ (resp. $b_{2}$) and time $T_{0}> 0$ such that $\Phi_t(U)\thicksim \Phi_t(V)$ for all $t\geq T_{0}$. Choose some $\sigma>0$ so small that $\Phi_s(b_{1})\in U$ whenever $\abs{s-\Gamma(b_{1},b_{2})}\le\sigma$. It then follows that
  \begin{equation}\label{a-1a-2-or}
 \Phi_{T_{0}+\Gamma(b_{1},b_{2})-\sigma-t_{k}}(a_{1})\thicksim \Phi_{T_{0}-t_{k}}(a_{2})
 \end{equation}
 for $k$ sufficiently large. On the other hand, we have already known that  $\Gamma(b_{1},b_{2})=\sup_{t\ge 0} \{f(t)\}$ by the assertion above. So, one can find a $\tau >0$ such that $f(t)\ge \Gamma(b_{1},b_{2})-\sigma$ for $t\geq \tau$, which implies that $\Phi_{-t+\Gamma(b_{1},b_{2})-\sigma}(a_{1})\rightharpoondown \Phi_{-t}(a_{2})$ for all $t\geq \tau$, a contradiction to \eqref{a-1a-2-or}. Thus, we have proved $\Gamma(b_{1},b_{2})= +\infty$, which completes the proof of the lemma.
\end{proof}

%\begin{rmk} \indent For general  semi- discrete topological system on compact space ,we have find an example with homoclinic property from the paper of John Milnor[13], [14].\\
%\textit{\textbf{Example:}} Let $X$ be the circle of real number modulo $2\pi$, and let $f(\theta)=\theta+1-\cos(\theta)\;(mod\;2\pi)$. Then the semi-discrete system $(X,Z^{+},f)$  has the unique attractor $\theta=0$. And $\theta=0$ satisfies $\omega(\theta)=0=\alpha(\theta)$ for any $\theta\in X$.\\
%\end{rmk}

\begin{lem}\label{order-sqeu-1}
Let $O^{+}(x)$ be a nontrivial pseudo-ordered precompact semi-orbit. If there exist $a,b\in\omega(x)$ satisfying $a\thicksim b$ and $\overline{O(a)}\cap \overline{O(b)}=\emptyset$, then $B\ne\emptyset$.
  \end{lem}
  \begin{proof} Given any $z_{1}\in \omega(a)$, there is a sequence $t_{k}\to \infty$ such that $\Phi_{t_{k}}(a)\rightarrow z_{1}$. We may also assume without loss of generality that $\Phi_{t_{k}}(b)\rightarrow z_{2}\in \o(b)$. Because $\overline{O(a)}\cap \overline{O(b)}=\emptyset$, we have $z_{1}\thicksim z_{2}\neq z_{1}$.

  Let $\Gamma^{*}(z_{1},z_{2})=\sup\{t\geq 0: z_{1}\thicksim \Phi_{s}(z_{2})\, \text{for all}\, s\in [-t,t]\}$
and define the function $h(t)=\Gamma^{*}(\Phi_{t}(a),\Phi_{t}(b))$ for $t\ge 0$. Then it is not difficult to see that
$h(t)$ is nondecreasing on $[0,+\infty)$. Note also that $\Gamma^{*}(\Phi_{t_{k}}(a),\Phi_{t_{k}}(b))\leq \Gamma^{*}(z_{1},z_{2})$ for any $k\ge 0$. Then it follows that
\begin{equation}\label{tau*-}
\Gamma^{*}(\Phi_{t}(a),\Phi_{t}(b))\leq \Gamma^{*}(z_{1},z_{2})
\end{equation}
for any $t> 0$.

Suppose that $\Gamma^{*}(z_{1},z_{2})<+\infty$. Then we define $K=\{\Phi_{s}(z_{2}): s\in[-\Gamma^{*}(z_{1},z_{2}),$ $\Gamma^{*}(z_{1},z_{2})]\} $. Clearly, $K$ is compact. Noticing that $\overline{O(z_{1})}\cap \overline{O(z_{2})}=\emptyset$, one has $z_1\notin K$. Then there exists $U$ (resp. $V$) of $z_{1}$ (resp. $K$) and time $T>0$ such that $\Phi_{t}(U)\thicksim \Phi_{t}(V)$ for all $t\geq T$. Therefore, one can find an $N>0$ such that $\Phi_{t_{k}}(a)\in U$ and $\Phi_{t_{k}+s}(b)\in V$  for all $k\geq N$ and $s\in[-\Gamma^{*}(z_{1},z_{2}),\Gamma^{*}(z_{1},z_{2})]$. Fix such an $N>0$, there is an
 $\varepsilon> 0$ such that
$\Phi_{s+l+t_{N}}(b)\in V$ whenever $\abs{s}\le \Gamma^{*}(z_{1},z_{2})$ and $\abs{l}\le\varepsilon$.  Accordingly, we obtain that
$\Phi_{t_{N}+T}(a)\sim \Phi_{t_{N}+T+\Gamma^{*}(z_{1},z_{2})+\varepsilon}(b)$; and hence, $\Gamma^{*}(\Phi_{t_{N}+T}(a),\Phi_{t_{N}+T}(b))\geq\Gamma^{*}(z_{1},z_{2})+\varepsilon$, which contradicts \eqref{tau*-}.
Thus we have obtained  $\Gamma^{*}(z_{1},z_{2})=+\infty$, which implies that $z_{1}\thicksim \overline{O(z_{2})}$.

By virtue of Lemma \ref{parti-ordered}, we have $\overline{O(z_{2})}\thicksim \omega(x)$, i.e, $B\neq \emptyset$.
\end{proof}

\vskip 2mm

\begin{prop}\label{B-empty}
Let $O^{+}(x)$ be a nontrivial pseudo-ordered precompact semi-orbit and assume that {\rm (F)} in \eqref{funda-assum} holds.  Assume also that $B=\emptyset$. Then
 $\omega(x)\subset E$ is a non-ordered set, where $E$ is the set of all the equilibria of $\Phi_t$.
\end{prop}

\begin{proof}
 Clearly, $A\ne \emptyset$ because $B=\emptyset$. So the cardinality of $A$ is at least $2$ due to the condition (F) and the statement at the beginning of the proof of Proposition \ref{B-nonempty}.
  So, if $\overline{O(a)}$ and $\overline{O(b)}$ are any two distinct elements in $A$, then it follows from $B=\emptyset$ and Proposition \ref{property-of-A-B}(ii) that $\overline{O(a)}\cap \overline{O(b)}=\emptyset$. By Lemma \ref{order-sqeu-1}, one can further obtain that $a\rightharpoondown b$. (otherwise, $B\ne \emptyset$, a contradiction.)

{\it We now claim that: For any $a\in\omega(x)$ and any neighborhood $B_{\e}(a)$ of $a$,
$B_{\e}(a)\cap \omega(x)\nsubseteq \overline{O(a)}$.}
 Before we give the proof of the claim, we first show how it implies our conclusion. Suppose that one can find an $a\in \omega(x)\setminus E$. Then there exists $T>0$ such that $a\neq \Phi_{T}(a)$. By Theorem A, $a\thicksim \Phi_{T}(a)$. Then, one can find a neighborhood $V$ of $\Phi_{T}(a)$ such that $\Phi_t(a)\thicksim \Phi_t(V)$ for all $t\ge T_{1}$. Therefore, the above claim will imply that there is a $d\in V\cap \omega(x)$ with $d\notin \overline{O(\Phi_T(a))}$, and hence, $\overline{O(d)}\cap\overline{O(a)}=\emptyset$. Together with $B=\emptyset$, Lemma \ref{order-sqeu-1} entails that $\Phi_{T_1}(a)\rightharpoondown\Phi_{T_1}(d)$. This contradicts $\Phi_{T_1}(U)\sim \Phi_{T_1}(V)$. Thus, we have proved that
  $\omega(x)\subset E$. Again by Lemma \ref{order-sqeu-1}, one can further obtain that $\omega(x)$ is non-ordered, which are the statements in this Proposition. So, in order to complete the proof of this Proposition, it suffices to prove the claim above.
  \vskip 2mm

{\it Proof of the claim:}  We first point out that $\overline{O(a)}$ is a minimal set for any $a\in\omega(x)$. Indeed, for any point $a\in\omega(x)$, the fact $B=\emptyset$ implies that $\overline{O(a)} \in A$. Suppose that $\overline{O(a)}$ is not minimal. Then one can find some point $c\in \overline{O(a)}$ such that $\overline{O(a)}\setminus \overline{O(c)}\ne \emptyset$. Together with Lemma \ref{parti-ordered}, it yields that $\overline{O(c)}\in B$, contradicting $B=\emptyset$.

 Now we suppose that the claim is not correct. Then there exists $\sigma>0$ such that $B_{\s}(a)\cap \omega(x)\subseteq\overline{O(a)}$. Recall that $\overline{O(a)}$ is a minimal set. Then one can find a positive integer $N(a,\sigma)$ satisfying that: For any fixed $q>0$, there exists an integer $r>0$ such that
\begin{equation}\label{aa-vva}
q\leq r\leq q+N(a,\sigma)\,\,\textnormal{ and }\,\,\Phi_r(a)\in B_{\frac{\sigma}{2} }(a).
\end{equation}
%{\blue Although the conclusion above is well-known for semi system $(X,\mathbb{G},T)$, where $\mathbb{G}$ is a semi-group, we think it's necessary to give a proof of it. Because the orbit of a point $a$ in our paper, which is different from the orbit $orb(x,\mathbb{G})=\{gx:g\in \mathbb{G}\}$ of semi system $(X,\mathbb{G},T)$, {\red consists of the positive semi-orbit $O^{+}(a)$ and a branch $O^{-}(a)$ of negative semi-orbit respect to $a$.} Now, suppose \eqref{aa-vva} is't true. Then there exists a increased positive number sequence $\{r_{i}\}_{1}^{+\infty}$ satisfying that $r_{i}\rightarrow +\infty$ as $i\rightarrow +\infty$ and $\Phi_s(a)\notin B_{a}(\sigma/2)$ for any $s\in [r_{i},r_{i}+i]$. By the compactness of $\overline{O(a)}$, we suppose $\lim\limits_{i\rightarrow+\infty}\Phi_{r_{i}}(a)=y\notin B_{a}(\sigma/2)$. Moreover, we have $\Phi_s(y)=\lim\limits_{i\rightarrow+\infty}\Phi_{r_{i}+s}(a)$ is not in $B_{a}(\sigma/2)$. It implies that $\omega(y)\subseteqq \overline{O(a)}\setminus B_{a}(\sigma/2)$, contradicting that $\overline{O(a)}$ is a minimal set.}
 By virtue of \eqref{aa-vva}, we can choose an increasing positive number sequence $\{t_{n}\}_{n=1}^{\infty}$ satisfying (i) $\Phi_{t_{n}}(a)\in B_{\frac{\sigma}{2} }(a)$; (ii) $\abs{ t_{n+1}-t_{n}}\leq N(a,\sigma)$. Then for any fixed $t>0$, one can take some $n>0$
 such that $t_{n}\leq t<t_{n+1}$; and hence,
 $$\Phi_t(a)\in \Phi_{t-t_{n+1}}(\Phi_{t_{n+1}}(a))\subset \Phi_{t-t_{n+1}}(B_{\frac{\sigma}{2} }(a))\subset \bigcup_{s\in[0,N(a,\sigma)]} \Phi_{-s}(B_{\frac{\sigma}{2} }(a)).$$
Therefore, we will further see that $\bigcup_{s\in[0,N(a,\sigma)]}\Phi_{-s}(B_{\sigma}(a))$ is an open cover of $\overline{O^+(a)}$.
Indeed, given any $b\in \overline{O^{+}(a)}$, choose a sequence $s_{i}>0$ such that $\Phi_{s_{i}}(a)\rightarrow b$. Since $a$ is a minimal point, there exists $\tau_{i}$ such that $s_{i}<\tau_{i}<s_{i}+N(a,\frac{\sigma}{2})$ and $\Phi_{\tau_{i}}(a)\in B_{\frac{\sigma}{2} }(a)$. For each $i$, write $\tau_{i}=s_{i}+r_{i}$ for some $r_{i}\in [0,N(a,\frac{\sigma}{2})]$. Without loss of generality, assume that $r_{i}\rightarrow r\in [0,N(a,\frac{\sigma}{2})]$. Then $\Phi_{\tau_{i}}(a)\rightarrow \Phi_{r}(b)$, which implies that $\Phi_{r}(b)\in\overline{B_{\frac{\sigma}{2}}(a)}\subseteq B_{\sigma}(a)$. Consequently, $\cup_{s\in[0,N(a,\sigma)]}\Phi_{-s}(B_{\sigma}(a))$ is an open cover of $\overline{O^{+}(a)}$. Noticing that $\overline{O^{+}(a)}=\overline{O(a)}$ (because $\overline{O(a)}$ is a minimal), it follows that $\cup_{s\in[0,N(a,\sigma)]}\Phi_{-s}(B_{\sigma}(a))$ is an open cover of $\overline{O(a)}$.

Recall that $B_{\sigma}(a)\cap \omega(x)\subseteq \overline{O(a)}$. It then follows that $$\left( \bigcup_{s\in[0,N(a,\sigma)]}\Phi_{-s}(B_{\sigma}(a))\right)\bigcap\omega(x)\subseteq\overline{O(a)}$$ (Otherwise, there exists a point $d \in \Phi_{-s}(B_{\sigma}(a))\cap \omega(x) $ such that $d\notin \overline{O(a)}$; and hence, $\overline{O(a)}\cap\overline{O(d)}=\emptyset$. Then, one has $\Phi_s(d)\in B_{\sigma}(a)\cap \omega(x)$ and
$\overline{O(\Phi_s(d))}\cap\overline{O(a)}=\emptyset$, a contradiction to  $B_{\sigma}(a)\cap \omega(x)\subseteq \overline{O(a)}$).

On the other hand, it is clear that $\overline{O(a)}\subseteq\left( \bigcup_{s\in[0,N(a,\sigma)]}\Phi_{-s}(B_{\sigma}(a))\right)\bigcap\omega(x)$. So,
\begin{equation}\label{good-1}
\left( \bigcup_{s\in[0,N(a,\sigma)]}\Phi_{-s}(B_{\sigma}(a))\right)\bigcap\omega(x)=\overline{O(a)}.
\end{equation}
As a consequence, we obtain that $\omega(x)\nsubseteq \bigcup_{s\in[0,N(a,\sigma)]}\Phi_{-s}(B_{\sigma}(a))$ (Otherwise, $\omega(x)=\overline{O(a)}$, contradicting \textbf{(F)} in \eqref{funda-assum}).

Since $\bigcup_{s\in[0,N(a,\sigma)]}\Phi_{-s}(B_{\sigma}(a))$ is an open cover of $\overline{O(a)}$, it follows from \eqref{good-1} that one can find a smaller open set $U$ such that $\overline{O(a)}\subset U\subsetneq \bigcup_{s\in[0,N(a,\sigma)]}$ $\Phi_{-s}(B_{\sigma}(a))$ and $U\cap \omega(x)=\overline{O(a)}$. Noticing that $\omega(x)\nsubseteq \bigcup_{s\in[0,N(a,\sigma)]}\Phi_{-s}(B_{\sigma}(a))$, we have obtained
 a contradiction to the connectivity of $\omega(x)$. Thus, we have completed the proof.
\end{proof}

\vskip 2mm
Now we are ready to prove Theorems B and D.
\vskip 2mm

{\it Proof of Theorem B.}
If (F) in \eqref{funda-assum} does not hold, i.e., $\overline{O(a)}=\omega(x)$ for some $a\in \o(x)$, it then directly follows from Theorem A that $\o(x)$ is ordered. If (F) in \eqref{funda-assum} holds, then Propositions \ref{B-nonempty} and \ref{B-empty} will imply the Trichotomy listed in Theorem B.

  In particular, if $\omega(x)\cap E=\emptyset$, then case (ii) can not happen. Moreover, Proposition \ref{B-nonempty}(ii) also implies that case (iii) can not happen either. As a consequence,
     $\omega(x)$ must be ordered.

  Finally, if $C$ is completed, then we choose a codim-$k$ subspace $H^c\subset X$ such that $H^{c}\cap C=\{0\}$. Based on the definition of $C$, one can also find a dim-$k$ subspace $H\subset C$. Consequently, $H \cap H^{c}=\{0\}$; and moreover, $H\oplus H^{c}=X $.
Now take $\Theta :X \rightarrow H $ the linear projection onto $H$ along $H^{c}$. Since $\o(x)$ is now ordered, one has $\Theta(a)\neq \Theta(b)$ for any distinct $a,b\in\omega(x)$. (Otherwise, $a-b\in H^{c}\setminus\{0\}$.  This then implies that $a-b\notin C$, contradicting that $\o(x)$ is ordered). Denote by $\Theta_{\omega}:=\Theta|_{\omega(x)}$ the restriction of $\Theta$ to $\o(x)$. Clearly, $\Theta_{\omega}$ is an one-to-one map. Then one can repeat the argument in the proof of
      \cite[Theorem 3.17]{HirSm} to obtain that $\omega(x)$ is topologically conjugate to
a compact invariant set of a Lipschitz-continuous vector field in $\mathbb{R}^k$.
$\quad\square$
\vskip 3mm

{\it Proof of Theorem D.}
 Assume that $\omega(x)$ is not an equilibrium. Suppose also that there are two distinct points $p,q\in \omega(x)$ such that $p\thicksim q$. Then one can find open sets $U\ni p$, $V\ni q$ and time $T>0$ such that $\Phi_t(U)\thicksim \Phi_t(v)$ for $t>T$. Now choose $t_{1},t_{2}>0$ so large that $\Phi_{t_{1}}(x)\in U$, $\Phi_{t_{2}}(x)\in V$. So, one has $\Phi_{t_{1}}(x)\ne \Phi_{t_{2}}(x)$ and $\Phi_{t_{1}+T}(x)\thicksim\Phi_{t_{2}+T}(x)$, which contradicts that $O(x)$ is of Type-II.
$\qquad\qquad\qquad\qquad\square$

\section{Poincar\'{e}-Bendixson Theorem (Theorem C)}

In this section, we focus on the Poincar\'{e}-Bendixson type Theorem (Theorem \ref{P-B for semif} or Theorem C) for $2$-cones. As we mentioned in the introduction, Sanchez \cite{San09} obtained this Theorem by using the generalized Perron-Frobenius Theorem (see \cite{FO1}) and theory of invariant manifolds in $\mathbb{R}^n$, which strongly depend on the $C^1$-smoothness assumption of $\Phi_t$.
We will prove this Theorem on a Banach space $X$ without this smoothness assumption. Our approach is motivated by \cite{Smi95} and utilizes the chain-recurrent property of $\Omega$.

\begin{thm}\label{P-B for semif}
{\it Let {\rm (A1)-(A2)} hold. Assume that $k=2$ and $C$ is complemented. Let $O^{+}(x)$ be a nontrivial pseudo-ordered precompact semi-orbit. If $\omega(x)\cap E=\emptyset$, then $\omega(x)$ is a periodic orbit.}
\end{thm}
\begin{proof}
For brevity, we write $L=\omega(x)$. By virtue of Theorem B, $L$ is ordered and topologically conjugate to a flow $\Psi_t$ on the compact invariant set $\Theta(L)\subset\mathbb{R}^2$.
Clearly, $L$ is a chain-recurrent set. It is also not difficult to see that $\Theta(L)$ is a chain-recurrent set with respect to $\Psi_t$.

 Note that $\Theta(L)\subset \mathbb{R}^2$ and $\Theta(L)$ does not contain any equilibrium of $\Psi_t$ (because $L\cap E=\emptyset$). Then, by following the same argument in \cite[Theorem 4.1]{Smi95}, the chain-recurrent property of $\Theta(L)$ will imply that $\Theta(L)$ is either a single periodic-orbit or consists of an annulus of periodic orbits.

We now rule out the possibility of the case that $\Theta(L)$ is an annulus of periodic orbits. Suppose not, let $\g$ be a periodic orbit in $L$ such that the periodic orbit $\Theta(\g)\subset {\rm Int}(\Theta(L))\subset \mathbb{R}^2$. Then $\Theta(\g)$ separates $\Theta(L)$ into two components. Fix $a,b\in L$ such that $\Theta(a),\Theta(b)$ belong to the different component of $\Theta(L)\setminus \Theta(\g)$. Since $\{\Phi_t(x)\}_{t\ge 0}$ will repeatedly revisit the neighborhood of $a$ and $b$, $\Theta(\Phi_t(x))$ will intersect $\Theta(\g)$ at a sequence $t_k\to \infty$. So, one can choose $z_k\in \g, k=1,2,\cdots,$ such that $\Theta(\Phi_{t_k}(x))=\Theta(z_k)$, which entails that $\Phi_{t_k}(x)\rightharpoondown z_k$ for $k=1,2,\cdots.$ Therefore, one can obtain that, {\it for any $s>0$, there exists some $w_s\in \g$ such that $\Phi_s(x)\rightharpoondown w_s$.} (Otherwise, one can find some $s>0$ such that $\Phi_s(x)\sim \g$. Then one can choose some $t_k>s$, it follows from monotonicity of $\Phi_t$ that $\Phi_{t_k}(x)\sim \g$, contradicting $\Phi_{t_k}(x)\rightharpoondown z_k\in \g$.)

 Now, choose any
$y\in L\setminus \g$, there exists a sequence $s_n\to \infty$ such that $\Phi_{s_n}x\to y$ as $n\to \infty$. By the assertion above, one can also find $w_n\in \g$ such that $\Phi_{s_n}(x)\rightharpoondown w_n\in \g$. Without loss of generality, one can assume that $w_n\to w\in \g$ as $n\to \infty$. So, by letting $n\to \infty$, we have (by strongly monotone property)
$y\rightharpoondown w$ or $y-w\in \partial C\setminus \{0\}$. Noticing that $y\ne w$ and $y,w\in L$, we have obtained a contradiction to the fact that $L$ is strongly ordered (see the following Remark \ref{rem-f}).

Thus, we have rule out the possibility of the case that $\Theta(L)$ is an annulus of periodic orbits. As a consequence, $\Theta(L)$ is a single periodic-orbit. This immediately implies that $L$ is a single periodic-orbit. We have completed the proof.
\end{proof}

\vskip 2mm

\begin{rmk}\label{rem-f}
\textnormal{Due to the strong monotonicity (which is stronger than monotonicity), we have $\omega(x)$ is strongly ordered, i.e., for any two points $y,z\in \omega(x)$, $y-z\in {\rm int}C$.}
\end{rmk}

\section{Discussion and further results}

In this section, we will discuss the relationship between our work on strongly monotone systems with respect to $k$-cones and the established theory of several well-known systems.
\vskip 2mm

\noindent $\bullet$ {\it Competitive Dynamical Systems.} The well-known competitive dynamical systems on $\mathbb{R}^n$ (see \cite{HirSm,WJ} and references therein) can be viewed as monotone systems with respect to the $(n-1)$-cone $C$ whose complemented cone is of rank-$1$.

In the terminology of
strongly monotone systems with respect to $C$, one of the most remarkable results in strongly competitive systems can be reformulated as: {\it Any omega-limit set of competitive systems is ordered with respect to $C$ and is topologically conjugate to a compact flow in $\mathbb{R}^{n-1}$} (see, e.g. \cite{Smi95,WJ}), which is exactly the problem we discussed in this paper. Moreover, for competitive systems, the additional non-oscillation principle (see \cite{Hirsch1,Smi95,WJ}, due to the rank-1 property of the complemented cones) guarantees that any  omega-limit set is ordered with respect to $C$. While in our setting, since the complemented cone is not necessarily rank-$1$, one can only obtain the partial information of the ordering of the omega-limit sets as in our Theorems A, B and D.

\vskip 3mm
\noindent $\bullet$ {\it Systems with Quadratic Cones.}
\vskip 2mm

Consider a system
\begin{equation}\label{E:ODE-sys}
\dot{x}=F(x),\quad x\in S\subset \mathbb{R}^n,
\end{equation}
in which the function $F:S\to \mathbb{R}^n$ satisfies a local Lipschitz condition in an open subset $S\subset\mathbb{R}^n$.

Let $P$ be a constant real symmetric non-singular matrix $n\times n$ matrix, with $2$ negative eigenvalues and $(n-2)$ positive eigenvalues. Then the set
\begin{equation}\label{E:2-cone}
C^-(P)=\{x\in \mathbb{R}^n:x^*Px\le 0\}
\end{equation} is a $2$-solid cone which is also complemented. Here $x^*$ denote the transpose of the vector $x\in \mathbb{R}^n$.

When $F$ is of class $C^1$ in \eqref{E:ODE-sys}, Sanchez \cite{San09} proved the connections between R. Smith's results \cite{SmithR-2,SmithR-3} and the strongly monotone systems with respect to the quadratic cone $C^-(P)$. In our work here, we confirm this connection even for locally Lipschitz continuous vector fields $F$. In this sense, Theorem C in Section 2 can be viewed as a generalization of the Poincar\'{e}-Bendixson Theorem of R. Smith in \cite{SmithR-2,SmithR-3}.

\begin{prop}\label{P-B-thm-C-p}
 Assume that there is a real number $\lambda$ (not necessarily positive), such that
\begin{equation}\label{decay-11}
(x-y)^*\cdot P\cdot [F(x)-F(y)+\lambda (x-y)]<0
\end{equation}
for any $x,y\in S$. Then the flow generated by \eqref{E:ODE-sys} is strongly monotone w.r.t. $C^-(P)$; and hence, Theorem C holds for system \eqref{E:ODE-sys}.
\end{prop}
\begin{proof}
For such $\lambda$, define $V(x)=x^*Px, x\in \mathbb{R}^n$. A direct calculation yields that
\begin{eqnarray*}
&&\frac{d}{dt}[V(x(t)-y(t))+2\lambda V(x(t)-y(t))]\\
&=&2(x(t)-y(t))^*P\cdot [F(x(t))-F(y(t))+\lambda (x(t)-y(t))].
\end{eqnarray*}
Together with \eqref{decay-11}, this implies that the function $t\mapsto e^{2\lambda t}V(x(t)-y(t))$ is strictly decreasing for $t\ge 0$, whenever $x(t),y(t)\in S$.

Given any $x-y\in C^-(P)\setminus \{0\}$, one has $x\ne y$ and $V(x-y)\le 0$. It then follows that $e^{2\lambda t}V(x(t)-y(t))<V(x-y)\le 0$ for any $t>0$. This entails that $V(x(t)-y(t))<0$, that is, $x(t)-y(t)\in {\rm Int}C^-(P)$ for all $t>0$. the flow generated by \eqref{E:ODE-sys} is strongly monotone w.r.t. $C^-(P)$.
\end{proof}

\vskip 2mm
\begin{rmk}
\textnormal{ (i) In \cite{SmithR-2,SmithR-3}, R.A. Smith assumed that $F$ satisfies
\begin{equation}\label{decay-11-smith}
(x-y)^*\cdot P\cdot [F(x)-F(y)+\lambda (x-y)]\le -\epsilon\abs{x-y}^2
\end{equation}
for any $x,y\in S$, where $\lambda,\epsilon>0$ are positive constants and $\abs{x-y}$ denote the Euclidean norm of the vector $x-y$. Clearly, Condition \eqref{decay-11} is weaker than \eqref{decay-11-smith}.
As a consequence, our Proposition \ref{P-B-thm-C-p} (which is based on the theory of monotone dynamical systems w.r.t. $2$-cone) can be viewed as a generalization of the Poincar\'{e}-Bendixson theorem of R. Smith in \cite{SmithR-2,SmithR-3}.}

\textnormal{ (ii) Since there is no $C^1$-smoothness assumption in Condition \eqref{decay-11}, we have improved Proposition 7 in Sanchez \cite{San09}. As we mentioned in Remarks \ref{Remark-on-ThmA} and \ref{Remark-on-ThmC} in Section 2, the smoothness assumption plays a key role in the approaches in \cite{San09}.}

\textnormal{ (iii) If one assumes $F$ is of class $C^1$, by following the same arguments in Ortega and Sanchez \cite[Remarks 1-2]{OSan}, one may obtain that
\eqref{decay-11} holds if and only if
\begin{equation*}\label{decay-11-smooth}
PDF(x)+(DF(x))^*P+\lambda P<0\,\, \text{ for any }x\in S,
\end{equation*} where $DF(x)^*$ stands for the transpose of the Jacobian $DF(x)$ and $<$ represents the usual order in the space of symmetric matrices. On the other hand, (7) in \cite [Proposition 7]{San09} is equivalent to existence of a (continuous) function $\lambda:\mathbb{R}^n\to \mathbb{R}$ such that \begin{equation*}\label{decay-11-smooth-OSan}
PDF(x)+(DF(x))^*P+\lambda(x) P<0\,\, \text{ for any }x\in \mathbb{R}^n.
\end{equation*}}
\end{rmk}

\end{document}